\documentclass{amsart}

\newcommand{\inj}{\mathrm{inj}}
\newcommand{\st}{\,\big|\,}
\newcommand{\real}{\mathbb{R}}

\newtheorem{theorem}{Theorem}[section]
\newtheorem{lemma}[theorem]{Lemma}
\newtheorem{corollary}[theorem]{Corollary}

\numberwithin{equation}{section}

\begin{document}

\title[CONVEXITY RADIUS OF A RIEMANNIAN MANIFOLD]{THE CONVEXITY RADIUS OF A RIEMANNIAN MANIFOLD}

\author[JAMES DIBBLE]{JAMES DIBBLE}
\address{Department of Mathematics, Rutgers University--New Brunswick, 110 Frelinghuysen Road, Piscataway, NJ 08854}
\curraddr{Department of Mathematics, University of Iowa, 14 MacLean Hall, Iowa City, IA 52242}
\email{james-dibble@uiowa.edu}
\thanks{These results are included in my doctoral dissertation \cite{Dibble2014}. I would like to thank my dissertation advisor, Xiaochun Rong, for his guidance and support, Christopher Croke for suggesting revisions to an early draft of this paper, and Oguz Durumeric and the anonymous referee for each identifying statements in need of correction and clarification.}


\date{}

\begin{abstract}
    The ratio of convexity radius over injectivity radius may be made arbitrarily small within the class of compact Riemannian manifolds of any fixed dimension at least two. This is proved using Gulliver's method of constructing manifolds with focal points but no conjugate points. The approach is suggested by a characterization of the convexity radius that resembles a classical result of Klingenberg about the injectivity radius.
\end{abstract}

\maketitle

\section{Introduction}

A subset $X$ of a Riemannian manifold $M$ is \textbf{strongly convex} if any two points in $X$ are joined by a unique minimal geodesic $\gamma : [0,1] \rightarrow M$ and each such geodesic maps entirely into $X$. It is well known that there exist functions $\inj,r : M \rightarrow (0,\infty]$ such that, for each $p \in M$,
\[
    \inj(p) = \max \big\{ R > 0 \st \exp_p|_{B(0,s)} \textrm{ is injective for all } 0 < s < R \big\}
\]
and
\[
    r(p) = \max \big\{ R > 0 \st B(p,s) \textrm{ is strongly convex for all } 0 < s < R \big\}\textrm{,}
\]
where $B(0,s) \subset T_p M$ denotes the Euclidean ball of radius $s$ around the origin. The number $\inj(p)$ is the \textbf{injectivity radius at $p$}, and $r(p)$ is the \textbf{convexity radius at $p$}. The \textbf{conjugate radius at $p$} is defined, as is customary, to be
\begin{align*}
    r_c(p) = \min \big\{ T > 0 \st \exists \textrm{ a} &\textrm{ non-trivial normal Jacobi field } J \textrm{ along a unit-speed}\\
    &\textrm{geodesic } \gamma \textrm{ with } \gamma(0) = p \textrm{, } J(0) = 0 \textrm{, and } J(T) = 0 \big\}\textrm{,}
\end{align*}
and the \textbf{focal radius at $p$} is introduced here to be
\begin{align*}
    r_f(p) = \min \big\{ T > 0 \st \exists \textrm{ a} & \textrm{ non-trivial normal Jacobi field } J \textrm{ along a unit-speed}\\
    &\textrm{geodesic } \gamma \textrm{ with } \gamma(0) = p \textrm{, } J(0) = 0 \textrm{, and } \| J \|'(T) = 0 \big\} \textrm{.}
\end{align*}
Either of these is defined to be infinite if the corresponding Jacobi fields do not exist. Short arguments show that they are well defined and that $r_f(p) \leq r_c(p)$, with equality if and only if both are infinite.

If $\gamma : [a,b] \rightarrow M$ is a geodesic connecting $p$ to $q$, then $p$ is \textbf{conjugate to $q$ along $\gamma$} if there exists a non-trivial normal Jacobi field $J$ along $\gamma$ that vanishes at the endpoints. If $\sigma : I \rightarrow M$ is a geodesic and $\gamma : [a,b] \rightarrow M$ is a geodesic connecting $p$ to $\sigma(s)$, where $I$ is an interval and $s \in I$, then $p$ is \textbf{focal to $\sigma$ along $\gamma$} if there exists a non-trivial normal Jacobi field $J$ along $\gamma$ such that $J(a) = 0$ and $J(b) = \sigma'(s)$. Conjugate and focal points correspond to singularities of the restriction of the exponential map to $T_q M$ and the normal bundle of $\sigma(I)$, respectively. Employing arguments similar to the proof of Proposition 4 in \cite{O'Sullivan1974}, one finds that the conjugate radius at $p$ is the length of the shortest geodesic $\gamma : [a,b] \rightarrow M$ along which $p$ is conjugate to $\gamma(b)$, while the focal radius at $p$ is the length of the shortest geodesic $\gamma : [a,b] \rightarrow M$ along which $p$ is focal to a non-constant geodesic normal to $\gamma$ at $\gamma(b)$. Let $\inj(M) = \inf_{p \in M} \inj(p)$, and similarly define numbers $r(M)$, $r_c(M)$, and $r_f(M)$. It's widely known that, when $M$ has sectional curvature bounded above, the first three are positive. The results in the third section of \cite{Eschenburg1987} imply the same about the focal radius.

When $M$ is compact, it was shown by Berger \cite{Berger1976} that $r(M) \leq \frac{1}{2}\inj(M)$. Berger has also pointed out that there are no examples in the literature where this inequality is known to be strict \cite{Berger2003}. It will be shown in this paper that $\inf \frac{r(M)}{\inj(M)} = 0$ over the class of compact manifolds of any fixed dimension at least two. The proof is suggested by alternative characterizations of the injectivity and convexity radiuses. Klingenberg \cite{Klingenberg1959} showed that $\inj(M) = \min \big\{ r_c(M),\frac{1}{2}\ell_c(M) \big\}$, where $\ell_c(M)$ is the length of the shortest non-trivial closed geodesic in $M$. It will be shown here that $r(M) = \min \big\{ r_f(M), \frac{1}{4}\ell_c(M) \big\}$. To the best of my knowledge, this equality does not appear elsewhere in the literature. Gulliver \cite{Gulliver1975} introduced a method of constructing compact manifolds with focal points but no conjugate points. For such manifolds, $r_f(M) < \infty$ and $r_c(M) = \infty$. The main result follows by showing that Gulliver's method may be used to construct manifolds with $\frac{r_f(M)}{\ell_c(M)}$ arbitrarily small.

\section{Geometric radiuses}

When $M$ is complete, each $v \in T M$ determines a geodesic $\gamma_v : (-\infty,\infty) \rightarrow M$ by the rule $\gamma_v(t) = \exp(tv)$. For each $p \in M$, the \textbf{cut locus at $p$} is the set
\begin{align*}
    \mathrm{cut}(p) = \big\{ v \in T_p M \st \gamma_v&|_{[0,T]} \textrm{ is minimal if and only if } T \leq 1  \big\}\textrm{,}
\end{align*}
and the \textbf{conjugate locus at $p$} is
\[
    \mathrm{conj}(p) = \big\{ v \in T_p M \st \exp_p : T_p M \rightarrow M \textrm{ is singular at } v \big\}\textrm{.}
\]
A \textbf{geodesic loop} is a geodesic $\gamma : [a,b] \rightarrow M$ such that $\gamma(a) = \gamma(b)$, while a \textbf{closed geodesic} additionally satisfies $\gamma'(a) = \gamma'(b)$. For each $p \in M$, denote by $\ell(p)$ the length of the shortest non-trivial geodesic loop based at $p$. Let $\ell(M) = \inf_{p \in M} \ell(p)$, and recall from the introduction that, for compact $M$, $\ell_c(M)$ equals the length of the shortest non-trivial closed geodesic in $M$. According to a celebrated theorem of Fet--Lyusternik \cite{FetLyusternik1951}, $0 < \ell_c(M) < \infty$. A general relationship between $\inj$ and $r_c$ is described by the following classical result of Klingenberg \cite{Klingenberg1959}.

\begin{theorem}[Klingenberg]\label{klingenberg lemma}
    Let $M$ be a complete Riemannian manifold. If $p \in M$ and $v \in \mathrm{cut}(p)$ has length $\inj(p)$, then one of the following holds:\\
        \indent (i) $v \in \mathrm{conj}(p)$; or\\
        \indent (ii) $\gamma_v|_{[0,2]}$ is a geodesic loop.\\
    \noindent Consequently, $\inj(p) = \min \big\{ r_c(p), \frac{1}{2}\ell(p) \big\}$.
\end{theorem}

\noindent Klingenberg used this to characterize $\inj(M)$.

\begin{corollary}[Klingenberg]\label{klingenberg corollary}
    The injectivity radius of a complete Riemannian manifold $M$ is given by $\inj(M) = \min \big\{ r_c(M),\frac{1}{2}\ell(M) \big\}$. When $M$ is compact, it is also given by $\inj(M) = \min \big\{ r_c(M),\frac{1}{2}\ell_c(M) \big\}$.
\end{corollary}

\noindent It's not clear that a pointwise result like that in Theorem \ref{klingenberg lemma} holds for the convexity radius, but global equalities like those in Corollary \ref{klingenberg corollary} will be proved.

The following lemma is an application of the second variation formula. Note that a $C^2$ function $f : M \rightarrow \real$ is \textbf{strictly convex} if its Hessian $\nabla^2 f$ is positive definite. This is equivalent to the condition that, for any non-constant geodesic $\gamma : (-\varepsilon,\varepsilon) \rightarrow M$, $(f \circ \gamma)''(0) > 0$.

\begin{lemma}\label{convexity of distance}
    Let $M$ be a Riemannian manifold and $p \in M$. If $R \leq r_f(p)$ and $\exp_p$ is defined and injective on $B(0,R) \subset T_p M$, then $d^2(p,\cdot) : B(p,R) \rightarrow [0,R^2)$ is strictly convex.
\end{lemma}

\noindent It will be useful to know that the convexity radius is globally bounded above by the focal radius. This may be proved using the following consequence of the Morse index theorem \cite{Morse1934}: If $\gamma$ and $\sigma$ are unit-speed geodesics, $\gamma(0) = p$, $\gamma(T) = \sigma(0)$, $T = d \big( p,\sigma(0) \big) < \inj(p)$, and $p$ is focal to $\sigma$ along $\gamma|_{[0,T]}$, then, for sufficiently small $s$ and $\varepsilon$ satisfying $0 < \varepsilon < s$, the ball $B \big( \gamma(-s),T + s - \varepsilon \big)$ is not strongly convex. This implies an inequality relating the focal and convexity radiuses near each point.

\begin{lemma}\label{convexity radius bounded by focal radius}
    Let $M$ be a complete Riemannian manifold. Then, for each $p \in M$, $\liminf_{x \rightarrow p} r(x) \leq r_f(p)$.
\end{lemma}

\noindent The global inequality $r(M) \leq r_f(M)$ follows immediately. One may also prove a global inequality relating the conjugate and focal radiuses.

\begin{lemma}\label{focal versus conjugate radiuses}
    If $M$ is a complete Riemannian manifold, then $r_f(M) \leq \frac{1}{2}r_c(M)$.
\end{lemma}

\begin{proof}
    Fix $\varepsilon > 0$, and let $p \in M$ be such that $r_c(p) < r_c(M) + \varepsilon$. Choose a unit-speed geodesic $\gamma : [0,r_c(p)] \rightarrow M$ satisfying $\gamma(0) = p$ and a non-trivial normal Jacobi field $J$ along $\gamma$ with $J(0) = 0$ and $J \big( r_c(p) \big) = 0$. Write $q = \gamma \big( r_c(p) \big)$. There must exist $0 < T < r_c(p)$ such that $\| J \|'(T) = 0$. If $T \leq \frac{1}{2}r_c(p)$, then $r_f(p) \leq \frac{1}{2}r_c(p)$. If $T \geq \frac{1}{2}r_c(p)$, then, by reversing the parameterizations of $\gamma$ and $J$, one finds that $r_f(q) \leq \frac{1}{2}r_c(p)$. In either case, $r_f(M) < \frac{1}{2}[r_c(M) + \varepsilon]$.
\end{proof}

\noindent It's now possible to prove global equalities for the convexity radius.

\begin{theorem}\label{global equality for convexity radius}
    The convexity radius of a complete Riemannian manifold $M$ is given by $r(M) = \min \big\{ r_f(M), \frac{1}{4}\ell(M) \big\}$. When $M$ is compact, it is also given by $r(M) = \min \big\{ r_f(M), \frac{1}{4}\ell_c(M) \big\}$.
\end{theorem}

\begin{proof}
    Assume for the sake of contradiction that $r(M) > \frac{1}{4}\ell(M)$. Since $\ell(M) < \infty$, one may set $\varepsilon = \frac{4}{5}[r(M) - \frac{1}{4}\ell(M)] > 0$. Let $\gamma : [0,1] \rightarrow M$ be a non-trivial geodesic loop with $L(\gamma) < \ell(M) + \varepsilon$. Then $B \big( \gamma(\frac{1}{4}),\frac{1}{4}L(\gamma) + \varepsilon \big)$ and $B \big( \gamma(\frac{3}{4}),\frac{1}{4}L(\gamma) + \varepsilon \big)$ are strongly convex, from which it follows that each of $\gamma|_{[0,\frac{1}{2}]}$ and $-\gamma|_{[\frac{1}{2},1]}$ is the unique minimal geodesic connecting $\gamma(0)$ to $\gamma(\frac{1}{2})$. This contradiction, together with Lemma \ref{convexity radius bounded by focal radius}, implies that $r(M) \leq \min \big\{ r_f(M),\frac{1}{4}\ell(M) \big\}$.

    Assume that there exists $p \in M$ such that $r(p) < \min \big\{ r_f(M),\frac{1}{4}\ell(M) \big\}$. Let $\varepsilon_i \rightarrow 0$ be a decreasing sequence such that each $B \big( p,r(p) + \varepsilon_i \big)$ is not strongly convex. Then there exist $x_i,y_i \in B \big( p,r(p) + \varepsilon_i \big)$ and minimal geodesics $\gamma_i : [0,1] \rightarrow M$ from $x_i$ to $y_i$ such that $\gamma_i([0,1]) \not\subset B \big( p,r(p) + \varepsilon_i \big)$. Define constants $0 \leq \delta_i < \varepsilon_i$ by
    \[
        r(p) + \delta_i  = \max \{ d(p,x_i),d(p,y_i) \}\textrm{,}
    \]
    and fix $t_i \in (0,1)$ such that $d \big( p,\gamma_i(t_i) \big) \geq r(p) + \varepsilon_i$. Let $(a_i,b_i)$ be the connected component of $\big\{ t \in (0,1) \st d \big( p,\gamma_i(t) \big) > r(p) + \delta_i \big\}$ containing $t_i$. Without loss of generality, replace $x_i$ and $y_i$ with $\gamma_i(a_i)$ and $\gamma_i(b_i)$, respectively, and $\gamma_i$ with $\gamma_i|_{[a_i,b_i]}$, reparameterizing the latter so that $\gamma_i(0) = x_i$ and $\gamma_i(1) = y_i$. Since $L(\gamma_i) \leq 2[r(p) + \varepsilon_1]$ for all $i$, one may, by passing to a subsequence, suppose without loss of generality that $x_i \rightarrow x \in \partial B \big( p,r(p) \big)$, $y_i \rightarrow y \in \partial B \big( p,r(p) \big)$, and $\gamma_i$ uniformly converges to a minimal geodesic $\gamma : [0,1] \rightarrow M$ from $x$ to $y$.

    Assume that $x = y$, and choose $\delta > 0$ such that
    \[
        r(p) + 3\delta < \min \Big\{ r_f(M),\frac{1}{4}\ell(M) \Big\} \leq \frac{1}{2} \min \Big\{ r_c(M),\frac{1}{2}\ell(M) \Big\} = \frac{1}{2} \inj(M)\textrm{.}
    \]
    Let $i$ be large enough that $x_i,y_i \in B(x,\delta)$. Then $L(\gamma_i) = d(x_i,y_i) < 2\delta$, so $\gamma_i([0,1]) \subset B \big( p,r(p) + 3\delta \big)$. Write $R = \min \{ r_f(p),\inj(p) \}$. Since
    \[
        d(p,x_i) = d(p,y_i) = r(p) + \delta_i < r(p) + 3\delta < R
    \]
    and $\gamma_i$ is not constant, it follows from Lemma \ref{convexity of distance} that $d \big( p,\gamma_i(t) \big) < r(p) + \delta_i$ for all $t \in (0,1)$. This is a contradiction. So $x \neq y$.

    Since $d(x,y) \leq 2r(p) < \inj(M)$, $\gamma$ is the unique minimal geodesic connecting $x$ to $y$. Let $w_i,z_i \in B \big( p,r(p) \big)$ be sequences such that $w_i \rightarrow x$ and $z_i \rightarrow y$. Then there exist unique minimal geodesics $\sigma_i : [0,1] \rightarrow M$ from $w_i$ to $z_i$, which satisfy $\sigma_i([0,1]) \subset B \big( p,r(p) \big)$. Since $\sigma_i \rightarrow \gamma$, one finds that $\gamma([0,1]) \subset B(p,R)$. Because $\gamma$ is not constant, Lemma \ref{convexity of distance} implies that $d \big( p,\gamma(t) \big) < r(p)$ for all $t \in (0,1)$. However, by construction, $d \big( p,\gamma(t) \big) \geq r(p)$ for all $t \in [0,1]$. It follows from this contradiction that $r(M) = \min \big\{ r_f(M),\frac{1}{4}\ell(M) \big\}$.

    In the case that $M$ is compact, $\ell_c(M) \leq \ell(M)$, so $r(M) \leq \min \big\{ r_f(M),\frac{1}{4}\ell_c(M) \big\}$. Since $\inj(M) = \min \big\{ r_c(M),\frac{1}{2}\ell_c(M) \big\}$, the argument in the preceding three paragraphs shows, essentially without modification, that $r(M) = \min \big\{ r_f(M),\frac{1}{4}\ell_c(M) \big\}$.
\end{proof}

\section{Construction of compact manifolds with $\frac{r(M)}{\inj(M)}$ arbitrarily small}

According to the characterizations of the injectivity and convexity radiuses in the preceding section, $r(M) = \frac{1}{2}\inj(M)$ whenever $r_f(M) = \frac{1}{2}r_c(M)$. Gulliver's examples of compact manifolds with focal points but no conjugate points show that this latter equality may fail to hold \cite{Gulliver1975}.

\begin{theorem}[Gulliver]\label{gulliver}
    Let $(M,g)$ be a compact Riemannian manifold with constant sectional curvature $-1$. Suppose $p \in M$ satisfies $\inj(p) \geq 1.7$. Then there exists a Riemannian metric $h$ on $M$ that agrees with $g$ except on a $g$-ball $B_R = B(p,R)$ of radius $R = 1.7$ and that satisfies the following:

    \vspace{2pt}

    (i) $r_c(M,h) = \infty$; and

    \vspace{2pt}

    (ii) $r_f(B_R,h|_{B_R}) < \infty$.

    \vspace{2pt}

    \noindent The Riemannian manifold $(B_R,h|_{B_R})$ depends only on the dimension of $M$.
\end{theorem}

\noindent Gulliver's construction is to write $B_R$ as the union of a $g$-ball $B_r$ and an annulus $B_R \setminus B_r$, change the metric on $B_r$ to have constant curvature $(0.55)^2$, where $B_r$ is large enough that it contains focal points but no conjugate points, and interpolate between the metrics on $B_r$ and $M \setminus B_R$ through a radially symmetric metric on $B_R$. Provided $\inj(p) \geq 1.7$, this can be done without introducing conjugate points.

It will be useful to know that the fundamental group of a connected hyperbolic manifold is \textbf{residually finite}, which means that, for any non-trivial $[\gamma] \in \pi_1(M,p)$, there is a normal subgroup $G$ of $\pi_1(M,p)$ of finite index such that $[\gamma] \not\in G$. This is a special case of the following theorem of Mal'cev \cite{Mal'cev1940}, also sometimes attributed to Selberg \cite{Selberg1960}. Note that a group is \textbf{linear} if it is isomorphic to a subgroup of the matrix group $\mathrm{GL}(F,n)$ for some field $F$.

\begin{theorem}[Mal'cev]\label{mal'cev theorem}
    Every finitely generated linear group is residually finite.
\end{theorem}

\noindent If $M$ is compact and has a hyperbolic metric, then, for each $C > 0$, there exist only finitely many non-trivial closed geodesics $\{ \gamma_1,\ldots,\gamma_k \}$ in $M$ of length less than twice $C$ (see Theorem 12.7.8 in \cite{Ratcliffe2006}). For each corresponding $[\gamma_i] \in \pi_1(M,q_i)$, there exists a normal subgroup $G_i$ of $\pi_1(M,q_i)$ of finite index such that $[\gamma_i] \not\in G_i$. Each $G_i$ is identified with a unique finite-index subgroup of $\pi_1(M,p)$ via conjugation by any path connecting $p$ to $q_i$. Letting $G = \cap_{i=1}^k G_i$, one obtains a finite-index normal subgroup of $\pi_1(M,p)$ that does not contain, up to conjugation, any of the $[\gamma_i]$. Therefore, all non-trivial closed geodesics in the finite covering space $\tilde{M} = H^n / G$ have length at least twice $C$. Since $r_c(\tilde{M}) = \infty$, an application of Corollary \ref{klingenberg corollary} proves the following result, which is well known to hyperbolic geometers.

\begin{lemma}\label{large hyperbolic manifolds}
    For each $n \geq 2$ and $C > 0$, there exists a compact $n$-dimensional Riemannian manifold $M$ that has constant sectional curvature $-1$ and satisfies $\inj(M) \geq C$.
\end{lemma}

\noindent It may now be shown that Gulliver's construction can produce compact manifolds $M$ of any dimension $n \geq 2$ with $\frac{r(M)}{\inj(M)}$ arbitrarily small.

\begin{theorem}
    For each $n \geq 2$ and $\varepsilon > 0$, there exists a compact $n$-dimensional Riemannian manifold $M$ with $\frac{r(M)}{\inj(M)} < \varepsilon$.
\end{theorem}

\begin{proof}
    Let $D < \infty$ denote the focal radius of the $n$-dimensional manifold in part (ii) of Theorem \ref{gulliver}. According to Lemma \ref{large hyperbolic manifolds}, there exists a compact $n$-dimensional Riemannian manifold $(M,g)$ that has constant sectional curvature $-1$ and satisfies $\inj(M,g) > \max \big\{ 2R, \frac{D}{\varepsilon} + R \big\}$, where $R = 1.7$. Apply Gulliver's construction to produce a metric $h$ on $M$ that agrees with $g$ except on a $g$-ball $B_R = B(p,R)$, has no conjugate points, and satisfies $r_f(M,h) < D$. By Corollary \ref{klingenberg corollary} and Lemma \ref{convexity radius bounded by focal radius}, $\inj(M,h) = \frac{1}{2}\ell_c(M,h)$ and $r(M,h) \leq r_f(M,h) < D$.

    Let $\gamma : [0,1] \rightarrow M$ be a non-trivial closed $h$-geodesic. If $\gamma([0,1]) \cap B_R = \emptyset$, then $L_h(\gamma) = L_g(\gamma) \geq 2\inj(M,g)$. If $\gamma([0,1]) \cap B_R \neq \emptyset$, then one may suppose without loss of generality that $\gamma(0) \in B_R$. Since $(M,h)$ has no conjugate points, $[\gamma] \neq 0$, which implies the existence of $t_0 \in (0,1)$ such that $d_g \big( p,\gamma(t_0) \big) > R$. Let $(a,b)$ be the connected component of $\big\{ t \in (0,1) \st d_g \big( p,\gamma(t) \big) > R \big\}$ containing $t_0$. Because $R < r(M,g)$, there exists a unique minimal geodesic $\sigma : [0,1] \rightarrow M$ of $(M,g)$ connecting $\gamma(a)$ to $\gamma(b)$, which satisfies $\sigma([0,1]) \subset B_R$. Note that $L_g(\sigma) \leq 2R$. Since $(M,g)$ has no conjugate points, the concatenation $\gamma|_{[a,b]} \cdot \sigma^{-1}$ is homotopically non-trivial, which implies that $L_g(\gamma|_{[a,b]} \cdot \sigma^{-1}) \geq 2\inj(M,g)$. Hence
    \[
        L_h(\gamma) > L_g(\gamma|_{[a,b]}) \geq 2\inj(M,g) - 2R\textrm{.}
    \]
    It follows that $\inj(M,h) \geq \inj(M,g) - R$ and, consequently, that
    \[
        \frac{r(M,h)}{\inj(M,h)} < \frac{D}{\inj(M,g) - R} < \varepsilon\textrm{.}
    \]
\end{proof}

\bibliography{bibliography}

\providecommand{\bysame}{\leavevmode\hbox to3em{\hrulefill}\thinspace}
\providecommand{\MR}{\relax\ifhmode\unskip\space\fi MR }
\providecommand{\MRhref}[2]{%
  \href{http://www.ams.org/mathscinet-getitem?mr=#1}{#2}
}
\providecommand{\href}[2]{#2}
\begin{thebibliography}{10}

\bibitem{Berger1976}
Marcel Berger, \emph{Some relations between volume, injectivity radius, and
  convexity radius in {R}iemannian manifolds}, Differential geometry and
  relativity, Mathematical Phys. and Appl. Math., vol.~3, Reidel, Dordrecht,
  1976, pp.~33--42.

\bibitem{Berger2003}
\bysame, \emph{A panoramic view of {R}iemannian geometry}, vol.~II,
  Springer--Verlag, Berlin, 2003, xxiv+824 pp.

\bibitem{Dibble2014}
James Dibble, \emph{Totally geodesic maps into manifolds with no focal points},
  Ph.D. thesis, Rutgers University--New Brunswick, 2014, vii+136 pp.

\bibitem{Eschenburg1987}
Jost-Hinrich Eschenburg, \emph{Comparison theorems and hypersurfaces},
  Manuscripta Math. \textbf{59} (1987), 295--323.

\bibitem{FetLyusternik1951}
Abram Fet and Lazar Lyusternik, \emph{Variational problems on closed
  manifolds}, Doklady Akad. Nauk SSSR (N.S.) \textbf{81} (1951), 17--18
  (Russian).

\bibitem{Gulliver1975}
Robert Gulliver, \emph{On the variety of manifolds without conjugate points},
  Trans. Amer. Math. Soc. \textbf{210} (1975), 185--201.

\bibitem{Klingenberg1959}
Wilhelm Klingenberg, \emph{Contributions to {R}iemannian geometry in the
  large}, Ann. of Math. (2) \textbf{69} (1959), 654--666.

\bibitem{Mal'cev1940}
Anatoli\u{\i} Mal'cev, \emph{On isomorphic matrix representations of infinite
  groups of matrices}, Rec. Math. [Mat. Sbornik] N.S. \textbf{8} (1940),
  405--422 (Russian).

\bibitem{Morse1934}
Marston Morse, \emph{The calculus of variations in the large}, American
  Mathematical Society Colloquium Publications, vol.~18, American Mathematical
  Society, Providence, 1934, ix+368 pp.

\bibitem{O'Sullivan1974}
John~J. O'Sullivan, \emph{Manifolds without conjugate points}, Math. Ann.
  \textbf{210} (1974), 295--311.

\bibitem{Ratcliffe2006}
John Ratcliffe, \emph{Foundations of hyperbolic manifolds}, second ed.,
  Graduate Texts in Mathematics, vol. 149, Springer, New York, 2006, xii+779
  pp.

\bibitem{Selberg1960}
Atle Selberg, \emph{On discontinuous groups in higher-dimensional symmetric
  spaces}, Contributions to function theory (Internat. Colloq. Function Theory,
  Bombay, 1960), Tata Institute of Fundamental Research, Bombay, 1960,
  pp.~147--164.

\end{thebibliography}
\bibliographystyle{amsplain}

\end{document}